\documentclass[12pt]{article}
\usepackage{color}
\usepackage{amssymb}
\usepackage{amsthm}
\usepackage{amsmath}
\usepackage[all]{xy}
\usepackage{graphicx}
\usepackage{enumerate}

\newcommand{\mb}[1]{\mathbf{ #1}}
\newcommand{\mc}[1]{\mathcal{ #1}}

\newcommand{\alga}{\mathbf A}

\DeclareMathOperator{\DM}{DM}
\DeclareMathOperator{\BDM}{\mathbf{DM}}

\newtheorem{theorem}{Theorem}[section]
\newtheorem{definition}[theorem]{Definition}
\newtheorem{lemma}[theorem]{Lemma}
\newtheorem{proposition}[theorem]{Proposition}
\newtheorem{remark}[theorem]{Remark}
\newtheorem{example}[theorem]{Example}
\newtheorem{corollary}[theorem]{Corollary}
\title{Algebraic properties of paraorthomodular posets}
\author{Ivan~Chajda, Davide~Fazio, Helmut~L\"anger, Antonio~Ledda and \\
Jan~Paseka}
\date{}
\begin{document}

\maketitle

\begin{abstract}Paraorthomodular posets are bounded partially ordered set with an antitone involution induced by quantum structures arising from the logico-algebraic approach to quantum mechanics. The aim of the present work is starting a systematic inquiry into paraorthomodular posets theory both from an algebraic and order-theoretic perspective. On the one hand, we show that paraorthomodular posets are amenable of an algebraic treatment by means of a smooth representation in terms of bounded directoids with antitone involution. On the other, we investigate their order-theoretical features in terms of forbidden configurations. Moreover, sufficient and necessary conditions characterizing bounded posets with an antitone involution whose Dedekind-MacNeille completion is paraorthomodular are provided.
\end{abstract}



\section{Introduction}
It is well known (see e.g.\ \cite{GLP16}) that any physical theory determines a class of event-state systems $(\mc E,\mc S)$, where $\mc E$ contains the events that may occur relatively to a given system and $\mc S$ contains the states that such a physical system may assume. In quantum theory, it is customary to identify $\mc E$ with the set $\Pi (\mc H)$ of projection operators of a Hilbert space $\mc H$, which is in bijective correspondence with the set $\mc C(\mc H)$ of closed subspaces of $\mc H$. Moreover, if we order $\mc C(\mc H)$ by set inclusion, then it can be made into the universe of an \emph{orthomodular lattice} (OML). Orthomodular lattices were introduced in 1936 by G. Birkhoff and J. von Neumann \cite{BVN01}, and independently by  K.~Husimi \cite{Hus}, as a suitable algebraic tool for investigating the logical structure underlying physical theories which, like quantum mechanics, do not conform to classical logic.
However, later on it was recognized that this framework needed to be revised. In fact, joins (i.e.\ disjunctions in the corresponding logic) need not exist provided the elements in question are not orthogonal. For this reason, \emph{orthomodular posets} were defined with the aim of faithfully formalising event structures of quantum mechanical systems (see e.g.\ \cite{Be85, Ka83} for a detailed account). Since the early forties, a large amount of investigations on the algebraic and order theoretical properties of orthomodular posets and, later, \emph{orthoalgebras} (see below and e.g.\ \cite{FGR92}), has become the core subject of investigations on \emph{sharp quantum structures} \cite{DGG}. However, in spite of the importance and depth of results in this field, some doubts concerning the relevance of OML's  for the foundation of quantum mechanics arose when it was discovered that the lattices of projection operators on Hilbert spaces do not generate the whole variety of orthomodular lattices (see e.g.\ \cite{DGG}). This shows that there are equational properties of event-state systems that cannot be captured by the proposed mathematical abstraction. Hence, alternative approaches recently appeared. Let us mention e.g.\ \emph{weakly orthomodular} and \emph{dually weakly orthomodular} lattices \cite{CL01}, and \emph{pseudo-orthomodular} posets \cite{CLRepo, CLP21}. 

A different abstract counterpart of the set of quantum events has been suggested within the  \emph{unsharp} approach to quantum theory by effect algebras \cite{FoBe,GG01} and quantum MV algebras \cite{G01}. These structures ensure the algebraic treatment of quantum effects, namely bounded self-adjoint linear operators on a Hilbert space satisfying the Born's rule. Effect algebras framework is indeed very general: any orthoalgebra, and so any orthomodular poset (lattice) is an example thereof. However, since the \emph{canonical} order (CO) induced by the Born's rule (see \cite[Chapter 4]{DGG}) on the set $\mc E(\mc H)$ of effects on a Hilbert space $\mc H$ does not ensure that meets and joins of elements always exist (see e.g.\ \cite[Exercise 17, p.124]{dvurec}), effect algebras do not lend themselves to easy logico-algebraic  investigations. Therefore, in \cite{GLP16,GLP17} \emph{paraorthomodular lattices}, i.e.\ regular bounded lattices with an antitone involution $'$ satisfying the paraorthomodular condition 
\[
x\leq y\quad\text{and}\quad x'\land y=0\quad \Rightarrow\quad x=y
\]
were introduced. They abstract a natural generalization of the lattice ordering on closed subspaces of a Hilbert space to the whole class of effects by means of the Olson \emph{spectral} order $\leq_{s}$ \cite{Ol01}. This ordering properly extends (CO), since e.g.\ the condition $$E\leq_{s} F\ \mbox{iff}\ F-E\geq_{s} \mathbb{O}$$ need not be satisfied (see \cite{Ol01}). 
Although their ``concrete'' representatives enjoy several properties which are not shared by effect algebras, paraorthomodular lattices are nevertheless capable of capturing an essential cross-cutting trait of quantum structures. Indeed, as it will be clear in the sequel, if the notion of paraorthomodular lattice is relaxed by not requiring that the underlying bounded poset with antitone involution is either a lattice or regular, then the derived concept of \emph{paraorthomodular poset} frames quantum structures into a common unifying picture. Any effect algebra determines a bounded poset with antitone involution, which is paraorthomodular in the above sense.

The aim of the present work is starting a systematic inquiry into paraorthomodular posets from a dual perspective. On the one hand, we show that paraorthomodular posets are amenable of an algebraic tratment by means of a smooth representation in terms of bounded directoids with antitone involution. On the other, we investigate their order-theoretical features in terms of forbidden configurations. Moreover, sufficient and necessary conditions characterizing bounded posets with an antitone involution whose Dedekind-MacNeille completion is paraorthomodular are provided. These results generalize well known achievements in the framework of completions of quantum structures \cite{Finch, Fi70a}.

The paper is organized as follows. In Section \ref{sec:1}, after dispatching all the notions needed in the sequel, we characterize paraorthomodular posets in terms of forbidden configurations. Moreover, we show that any effect algebra induces a bounded poset with antitone involution which is paraorthomodular. In Section \ref{sec:3}, we introduce the class of paraorthomodular directoids as the algebraic counterpart of paraorthomodular posets.  It will turn out that paraorthomodular directoids  form a proper super quasi-variety of paraorthomodular lattices. Furthermore, we show that sharp quantum structures, in particular orthoalgebras, can be represented as a proper subvariety of paraorthomodular directoids. This result extends the application of techniques developed in \cite{CK}. Finally, in Section \ref{sec:4} we investigate bounded posets with an antitone involution whose Dedekind-MacNeille completion are paraorthomodular lattices. 
\section{Paraorthomodular posets}\label{sec:1}

In this section we firstly recall several known but useful concepts.

A {\em paraorthomodular lattice} is a bounded lattice $\mathbf L=(L,\vee,\wedge,{}',0,1)$ with an antitone involution $'$ satisfying for any $x,y\in L$ the condition
\begin{enumerate}
\item[(P)] $x\leq y$ and $x'\wedge y=0$ imply $x=y$.
\end{enumerate}
Of course, every orthomodular lattice is paraorthomodular. In fact, if the antitone involution $'$ is a complementation, i.e.\ if $\mathbf L$ is an ortholattice, then $\mathbf L$ is orthomodular if and only if it satisfies condition (P). The matter is that in general $'$ need not be a complementation.

Let $\mathbf P=(P,\leq)$ be a poset, $a,b\in P$ and $A,B\subseteq P$. We define the lower and upper cone of $A$ as follows:
\begin{align*}
L(A) & :=\{x\in P\mid x\leq A\}, \\
U(A) & :=\{x\in P\mid x\geq A\}.
\end{align*}
Here $x\leq A$ means $x\leq y$ for all $y\in A$ and, similarly, $x\geq A$ means $x\geq y$ for all $y\in A$. The expression $A\leq B$ means $x\leq y$ for all $x\in A$ and $y\in B$. Instead of $L(\{a\})$, $L(\{a,b\})$ and $L(U(A))$ we simply write $L(a)$, $L(a,b)$ and $LU(A)$, respectively. Analogously, we proceed in similar cases.

The {\em poset} $\mathbf P$ is called {\em distributive} if it satisfies the LU-identity
\[
L(U(x,y),z)\approx LU(L(x,z),L(y,z))
\]
and {\em modular} if
\[
L(U(x,y),z)=LU(x,L(y,z))
\]
for all $x,y,z\in P$ with $x\leq z$.

Now we define our main concept.

\begin{definition}
A {\em paraorthomodular poset} is a bounded poset $\mathbf P=(P,\leq,{}',0,1)$ with an antitone involution $'$ satisfying for all $x,y\in P$ the condition
\begin{enumerate}
\item[{\rm(P*)}] $x\leq y$ and $L(x',y)=\{0\}$ imply $x=y$.
\end{enumerate}
\end{definition}
In the sequel, we will denote by $\mc P$ the class of paraorthomodular posets.	\\
Apparently, if $\mathbf P$ is a lattice then (P) is equivalent to (P*), thus both concepts coincide. On the other hand, there exist paraorthomodular posets which are not lattices.

It is easy to see that for a poset $(P,\leq,0)$ with bottom element $0$ and for $a,b\in P$ the assertions $L(a,b)=\{0\}$ and $a\wedge b=0$ are equivalent and that (P*) is equivalent to
\[
x<y\text{ implies }L(x',y)\neq\{0\}.
\]
Moreover, it is clear that if a poset $\mathbf P$ with an antitone involution is constructed from a bounded poset with an antitone involution by adding a new bottom element and a new top element, then $\mathbf P$ is paraorthomodular.

\begin{example}
\
\begin{enumerate}[{\rm(a)}]
\item The poset visualized in Fig.~1

\vspace*{-2mm}

\begin{center}
\setlength{\unitlength}{7mm}
\begin{picture}(6,8)
\put(3,1){\circle*{.3}}
\put(1,3){\circle*{.3}}
\put(3,3){\circle*{.3}}
\put(5,3){\circle*{.3}}
\put(1,5){\circle*{.3}}
\put(3,5){\circle*{.3}}
\put(5,5){\circle*{.3}}
\put(3,7){\circle*{.3}}
\put(3,1){\line(-1,1)2}
\put(3,1){\line(0,1)6}
\put(3,1){\line(1,1)2}
\put(1,3){\line(0,1)2}
\put(1,3){\line(1,1)2}
\put(1,3){\line(2,1)4}
\put(3,3){\line(-1,1)2}
\put(3,3){\line(1,1)2}
\put(5,3){\line(-2,1)4}
\put(5,3){\line(-1,1)2}
\put(5,3){\line(0,1)2}
\put(3,7){\line(-1,-1)2}
\put(3,7){\line(1,-1)2}
\put(2.85,.25){$0$} 
\put(.3,2.85){$a$}
\put(3.4,2.85){$b$}
\put(5.4,2.85){$c$}
\put(.3,4.85){$c'$}
\put(3.4,4.85){$b'$}
\put(5.4,4.85){$a'$}
\put(2.3,7.4){$1=0'$}
\put(2.2,-.75){{\rm Fig.\ 1}}
\end{picture}
\end{center}

\vspace*{4mm}

is paraorthomodular and modular, but not distributive since
\[
L(U(a,b),c)=L(c,a',b',c',1)=\{0,c\}\neq\{0\}=LU(0,0)=LU(L(a,c),L(b,c)).
\]
\item The poset visualized in Fig.~2

\vspace*{-2mm}

\begin{center}
\setlength{\unitlength}{7mm}
\begin{picture}(4,8)
\put(2,1){\circle*{.3}}
\put(1,3){\circle*{.3}}
\put(3,3){\circle*{.3}}
\put(1,5){\circle*{.3}}
\put(3,5){\circle*{.3}}
\put(2,7){\circle*{.3}}
\put(2,1){\line(-1,2)1}
\put(2,1){\line(1,2)1}
\put(1,3){\line(0,1)2}
\put(1,3){\line(1,1)2}
\put(3,3){\line(-1,1)2}
\put(3,3){\line(0,1)2}
\put(2,7){\line(-1,-2)1}
\put(2,7){\line(1,-2)1}
\put(1.85,.25){$0$}
\put(.3,2.85){$a$}
\put(3.4,2.85){$b$}
\put(.3,4.85){$b'$}
\put(3.4,4.85){$a'$}
\put(1.3,7.4){$1=0'$}
\put(1.2,-.75){{\rm Fig.\ 2}}
\end{picture}
\end{center}

\vspace*{4mm}

is paraorthomodular and distributive.
\item The poset visualized in Fig.~3

\vspace*{-2mm}

\begin{center}
\setlength{\unitlength}{7mm}
\begin{picture}(8,8)
\put(4,1){\circle*{.3}}
\put(1,3){\circle*{.3}}
\put(3,3){\circle*{.3}}
\put(5,3){\circle*{.3}}
\put(7,3){\circle*{.3}}
\put(1,5){\circle*{.3}}
\put(3,5){\circle*{.3}}
\put(5,5){\circle*{.3}}
\put(7,5){\circle*{.3}}
\put(4,7){\circle*{.3}}
\put(4,1){\line(-3,2)3}
\put(4,1){\line(-1,2)1}
\put(4,1){\line(1,2)1}
\put(4,1){\line(3,2)3}
\put(1,3){\line(0,1)2}
\put(1,3){\line(1,1)2}
\put(1,3){\line(2,1)4}
\put(3,3){\line(-1,1)2}
\put(3,3){\line(2,1)4}
\put(5,3){\line(-2,1)4}
\put(5,3){\line(1,1)2}
\put(7,3){\line(-2,1)4}
\put(7,3){\line(-1,1)2}
\put(7,3){\line(0,1)2}
\put(4,7){\line(-3,-2)3}
\put(4,7){\line(-1,-2)1}
\put(4,7){\line(1,-2)1}
\put(4,7){\line(3,-2)3}
\put(3.85,.25){$0$}
\put(.3,2.85){$a$}
\put(2.3,2.85){$b$}
\put(5.4,2.85){$c$}
\put(7.4,2.85){$d$}
\put(.3,4.85){$d'$}
\put(2.3,4.85){$c'$}
\put(5.4,4.85){$b'$}
\put(7.4,4.85){$a'$}
\put(3.3,7.4){$1=0'$}
\put(3.2,-.75){{\rm Fig.\ 3}}
\end{picture}
\end{center}

\vspace*{4mm}

is paraorthomodular, but not modular since $a\leq c'$, but
\begin{align*}
L(U(a,a'),c') & =L(1,c')=L(c')=\{0,a,d,c'\}\neq\{0,a,d\}=L(b',c',1)=LU(a,0,d)= \\
              & =LU(a,L(a',c')).
\end{align*}
\end{enumerate}
\end{example}

A forbidden configurations characterization of paraorthomodular posets comes next. 

The concept of a strong subposet was introduced by the first author and J.~Rach\r unek in \cite{CR}. Now we extend it for bounded posets with involution as follows.

\begin{definition}
A {\em strong subset} of a bounded poset $\mathbf P=(P,\leq,{}',0,1)$ with an antitone involution is a subset $A$ of $P$ containing $0$ and $1$ such that $A'\subseteq A$ and for all $x,y\in A$ the following hold:
\begin{itemize}
\item $U_{\mathbf P}L_{\mathbf A}(x,y)=U_{\mathbf P}L_{\mathbf P}(x,y)$,
\item $L_{\mathbf P}U_{\mathbf A}(x,y)=L_{\mathbf P}U_{\mathbf P}(x,y)$.
\end{itemize}
\end{definition}

That this concept is not curious is witnessed by the following result. It shows that strong subsets of a bounded poset with an antitone involution play a similar role as subsets do for lattices.

\begin{lemma}\label{lem:eux-forb conf}
Let $\mathbf P=(P,\leq,{}',0,1)$ be a bounded poset with an antitone involution $'$. If  $\mathbf M$ is a bounded subposet of $\mathbf P$ closed under\/ $'$ and $\mathbf M$ is a lattice, then $\mathbf M$ is a strong subset of $\mathbf P$.
\end{lemma}

\begin{proof}
Clearly, for any $x,y\in M$,
\begin{align*}
 x\lor^{\mathbf M}y & =x\lor^{\mathbf P}y, \\
x\land^{\mathbf M}y & =x\land^{\mathbf P}y.
\end{align*}
Therefore,
\begin{align*}
L_{\mathbf P}U_{\mathbf M}(x,y) & =L_{\mathbf P}U_{\mathbf M}(x\lor y)=L_{\mathbf  P}(x\lor y)=L_{\mathbf P}U_{\mathbf P}(x\lor y)=L_{\mathbf P}U_{\mathbf P}(x,y), \\
U_{\mathbf P}L_{\mathbf M}(x,y) & =U_{\mathbf P}L_{\mathbf M}(x\land y)=U_{\mathbf P}(x\land y)=U_{\mathbf  P}L_{\mathbf  P}(x\land y)=U_{\mathbf  P}L_{\mathbf  P}(x,y).
\end{align*}
\end{proof}

\begin{theorem}
Let $\mathbf  P=(P,\leq,{}',0,1)$ be a bounded poset with an antitone involution. Then $\mathbf  P$ is not paraorthomodular if and only if it contains a strong subposet $\mathbf  M$ ortho-isomorphic to $\mathbf B_6$ which is depicted in Fig.~4

\vspace*{-2mm}

\begin{center}
\setlength{\unitlength}{7mm}
\begin{picture}(4,8)
\put(2,1){\circle*{.3}}
\put(1,3){\circle*{.3}}
\put(3,3){\circle*{.3}}
\put(1,5){\circle*{.3}}
\put(3,5){\circle*{.3}}
\put(2,7){\circle*{.3}}
\put(2,1){\line(-1,2)1}
\put(2,1){\line(1,2)1}
\put(1,3){\line(0,1)2}
\put(3,3){\line(0,1)2}
\put(2,7){\line(-1,-2)1}
\put(2,7){\line(1,-2)1}
\put(1.85,.25){$0$}
\put(.3,2.85){$b'$}
\put(3.4,2.85){$a$}
\put(.3,4.85){$a'$}
\put(3.4,4.85){$b$}
\put(1.3,7.4){$1=0'$}
\put(1.2,-.75){{\rm Fig.\ 4}}
\end{picture}
\end{center}

\vspace*{4mm}

%\begin{figure}
%\begin{equation}
%  \xymatrix{
%&&1\ar@{-}[dl]\ar@{-}[dr]&&\\
%&a'\ar@{-}[d]&&b\ar@{-}[d]&\\
%&b'\ar@{-}[dr]&&a\ar@{-}[dl]&\\
%&&0&&
%}\tag{}
%\end{equation}
%\caption{the paraorthomodular lattice $\mathbf {B}_{6}$}\label{fig:b6}
%\end{figure}

\end{theorem}

\begin{proof}
Concerning the right-to-left direction, observe that $a\leq^{\mathbf P}b$ and
\[
L_{\mathbf P} (a',b)=L_{\mathbf P}U_{\mathbf P}L_{\mathbf P}(a',b)=L_{\mathbf P}U_{\mathbf P}L_{\mathbf M}(a',b)=L_{\mathbf P}U_{\mathbf P}(0)=\{0\},
\]
but $b\neq a$. As regards the left-to-right direction, suppose that $\mathbf P$ is not paraorthomodular. Then there exist $a,b\in P$ such that $a\leq b$ and $L(a',b)=\{0\}$, but $b\neq a$. Consider the subposet $\mathbf M:=(\{a,b,a',b',0,1\},\leq,{}',0,1)$ with the order inherited from $\mathbf P$. Note that $\mathbf M$ is a bounded sublattice of $\mathbf P$ with an antitone involution and, by Lemma~\ref{lem:eux-forb conf}, it is also a strong subset of $\mathbf P$.
\end{proof}

In the last part of this section, we show that paraorthomodular posets framework is indeed capable of setting quantum structures arising in the sharp and unsharp approach to quantum theory into a common unified picture.

Recall that an {\em orthomodular poset} is a bounded poset $(P,\leq,{}',0,1)$ with an antitone involution $'$ satisfying the following conditions for all $x,y\in P$:
\begin{enumerate}
\item[(1)] The unary operation $'$ is a complementation,
\item[(2)] $x\leq y'$ implies that $x\vee y$ exists,
\item[(3)] $x\leq y$ implies $y=x\vee(y\wedge x')$ ({\em orthomodular law}).
\end{enumerate}
Using (2) one can easily see that the expressions in (3) are well-defined, i.e.\ $y\wedge x'$ as well as $x\vee(y\wedge x')$ exist.

\begin{lemma}\label{lem: ompparaorthomod}
Every orthomodular poset is paraorthomodular.
\end{lemma}

\begin{proof}
If $(P,\leq,{}',0,1)$ is an orthomodular poset, $a,b\in P$, $a\leq b$ and $L(a',b)=\{0\}$ then $a'\wedge b$ exists and $a'\wedge b=0$ and using the orthomodular law we conclude
\[
a=a\vee0=a\vee(b\wedge a')=b.
\]
\end{proof}

Another interesting class of paraorthomodular posets are the so-called pseudo-or\-tho\-mod\-u\-lar ones, see \cite{CLRepo}.

A {\em pseudo-orthomodular poset} (\cite{CLRepo}) is a bounded poset $(P,\leq,{}',0,1)$ with an antitone involution $'$ which is a complementation satisfying one of the following equivalent conditions:
\begin{align*}
L(U(L(x,y),y'),y) & \approx L(x,y), \\
U(L(U(x,y),y'),y) & \approx U(x,y).
\end{align*}

\begin{proposition}\label{propseudo} 
Let $\mathbf P$ be a pseudo-orthomodular poset. Then $\mathbf P$ is paraorthomodular.
\end{proposition}

\begin{proof}
If $\mathbf P=(P,\leq,{}',0,1)$, $a,b\in P$, $a\leq b$ and $L(a',b)=\{0\}$ then
\[
U(a)=U(0,a)=U(L(b,a'),a)=U(L(U(b),a'),a)=U(L(U(b,a),a'),a)=U(b,a)=U(b)
\]
and hence $a=b$.
\end{proof}

We continue with the definition of an effect algebra (see \cite{FoBe} and \cite{dpos}). For an extensive account, the reader is referred to Dvure\v censkij and Pulmannov\'a's monograph \cite{dvurec}.

\begin{definition}\label{defea}
An {\em effect algebra} is a partial algebra $\mathbf E=(E,\oplus,0,1)$ of type $(2,0,0)$ satisfying conditions {\rm(E1)} -- {\rm(E4)} for all $x,y,z\in E$:
\begin{enumerate}[{\rm(E1)}]
\item If $x\oplus y$ exists, so does $y\oplus x$ and $x\oplus y=y\oplus x$,
\item if $x\oplus y$ and $(x\oplus y)\oplus z$ exist, so do $y\oplus z$ and $x\oplus(y\oplus z)$ and $(x\oplus y)\oplus z=x\oplus(y\oplus z)$,
\item there exists a unique $x'\in E$ such that $x\oplus x'$ is defined and $x\oplus x'=1$,
\item If $x\oplus1$ exists then $x=0$.
\end{enumerate}
\end{definition}

Since $'$ is a unary operation on $E$ it can be regarded as a further fundamental operation. Hence in the following we will write $\mathbf E=(E,\oplus,{}',0,1)$ instead of $\mathbf E=(E,\oplus,0,1)$.

Let $\mathbf E=(E,\oplus,{}',0,1)$ be an effect algebra and $a,b\in E$. The following facts are well-known:
\begin{enumerate}[(F1)]
\item By defining $a\leq b$ if there exists some $c\in E$ such that $a\oplus c$ exists and $a\oplus c=b$, $(E,\leq,{}',0,1)$ becomes a bounded poset with an antitone involution. We call $\leq$ the {\em induced order} of $\mathbf E$. Recall that the element $c$ is unique, if it exists. Then $c$ is equal to $(a\oplus b')'$ and it is denoted by $b\ominus a$.
\item $a\oplus b$ exists if and only if $a\leq b'$.
\item $a\oplus0$ and $0\oplus a$ exist and $a\oplus0=0\oplus a=a$.
\item $(a')'=a$.
\end{enumerate}

\begin{proposition}\label{prop} 
Every effect algebra $\mathbf P=(P,\oplus,0,1)$ is a paraorthomodular poset.
\end{proposition}

\begin{proof} It is enough to check condition (P*). Let $x,y\in P$ such that $x\leq y$ and $L(x',y)=\{0\}$. Then $L(y\ominus x)=L(x',y\ominus x)\subseteq L(x',y)$. Hence $y\ominus x=0$, i.e.\, $x=y$.
\end{proof}

\section{A representation by commutative directoids}\label{sec:3}

The aim of this section is to show that to every paraorthomodular poset $\mathbf P$ one can assign (in a non-unique way) a certain algebra $\mathbb D(\mathbf P)$ with everywhere defined operations such that the original poset $\mathbf P$ can be completely recovered from $\mathbb D(\mathbf P)$. Moreover, properties of $\mathbf P$ can often be expressed by means of identities and quasi-identities satisfied by $\mathbb D(\mathbf P)$.

The following concept was introduced by J.~Je\v zek and R.~Quackenbush, see \cite{JQ} and also \cite{CL}.

\begin{definition}
A {\em commutative directoid} is a groupoid $(D,\sqcap)$ satisfying the following identities
\begin{align*}
                   x\sqcap x & \approx x\text{ {\rm(}idempotence{\rm)}}, \\
                   x\sqcap y & \approx y\sqcap x\text{ {\rm(}commutativity{\rm)}}, \\
(x\sqcap(y\sqcap z))\sqcap z & \approx x\sqcap(y\sqcap z)\text{ {\rm(}weak associativity{\rm)}}.
\end{align*}
\end{definition}

Let $(P,\leq)$ be a downward directed poset (for example, a poset with $0$). Define $x\sqcap y:=x\wedge y$ if $x\wedge y$ exists and let $x\sqcap y=y\sqcap x$ be an arbitrary element of $L(x,y)$ otherwise ($x,y\in P$). It is elementary that $(P,\sqcap)$ is a commutative directoid. It will be called a {\em commutative directoid assigned} to $(P,\leq)$. Conversely, let $(D,\sqcap)$ be a commutative directoid. For all $x,y\in D$ define
\[
x\leq y\text{ if and only if }x\sqcap y=x.
\]
Then $(D,\leq)$ is a downward directed poset, the so-called {\em poset induced} by $(D,\sqcap)$. If $\mathbf {D}=(D,\sqcap)$ is obtained from a poset $\mathbf P$ as above, i.e.\ it is assigned to $\mathbf D$, we will express this fact also by $\mathbf D=\mathbb D(\mathbf P)$. Conversely, the fact that $\mathbf P$ is induced by $\mathbf D$ will be expressed also by $\mathbf P=\mathbb P(\mathbf D)$. If $\mathcal C$ and $\mathcal D$ is a class of posets and directoids, respectively, we will say that $\mathcal C$ is induced by $\mathcal D$ ($\mathcal D$ is assigned to $\mathcal C$) if
\begin{align*}
\mathcal C & =\{\mathbf P\mid\text{there exists some }\mathbf D\in\mathcal D\text{ with }\mathbb P(\mathbf D)=\mathbf P\}\text{ and} \\
\mathcal D & =\{\mathbf D\mid\text{there exists some }\mathbf P\in\mathcal C\text{ with }\mathbb D(\mathbf P)=\mathbf D\}.
\end{align*}
Although the assignment $(P,\leq)\mapsto(P,\sqcap)$ is not unique, the poset induced by $(P,\sqcap)$ coincides with the original one. Hence, every commutative directoid assigned to a poset contains the whole information concerning this poset.

Now let $\mathbf P=(P,\leq,{}')$ be a poset with an antitone involution and $\mathbb D(\mathbf P)$ an assigned commutative directoid. Then we can define
\[
x\sqcup y:=(x'\sqcap y')'
\]
for all $x,y\in P$. It is easy to show (by De Morgan's laws) that $(P,\sqcup)$ is a commutative directoid, too. Moreover,
\[
x\leq y\text{ if and only if }x\sqcup y=y.
\]
Recall from \cite S (see also \cite{CL}) that the algebra $(P,\sqcup,\sqcap)$ is called a {\em $\lambda$-lattice}.

\subsection{Paraorthomodular directoids}
If $\mathbf P=(P,\leq,{}',0,1)$ is an orthomodular poset then the assigned commutative directoids $\mathbb D(\mathbf P)=(P,\sqcap,{}',0,1)$ can be characterized by identities. Hence the class of these commutative directoids forms a variety, see \cite{CK}. A similar result holds for orthoposets. Analogously, to the class of complemented posets (satisfying additional properties) a variety of commutative directoids can be assigned, see \cite{CKL}. This motivates us to derive a similar result for the class of paraorthomodular posets. First we recall the following results from \cite{CK} and \cite{CKL}.

\begin{proposition}\label{prop1}
Let $\mathbf P=(P,\leq,{}')$ be a downward directed poset with a unary operation $'$ and $\mathbb D(\mathbf P)=(P,\sqcap,{}')$ an assigned commutative directoid. Then $'$ is an antitone involution on $(P,\leq)$ if and only if $\mathbb D(\mathbf P)$ satisfies the following identities:
\begin{enumerate}
\item[{\rm(4)}] $x''\approx x$,
\item[{\rm(5)}] $(x\sqcap y)'\sqcap y'\approx y'$.
\end{enumerate}
\end{proposition}

For the reader's convenience, we repeat the short proof.

\begin{proof}
Let $a,b\in P$. The operation $'$ is an involution if and only if (4) holds. If $'$ is antitone then from $a\sqcap b\leq b$ we obtain $b'\leq(a\sqcap b)'$, thus
\[
(a\sqcap b)'\sqcap b'=b',
\]
i.e.\ (5) holds. If, conversely, (5) holds and $a\leq b$ then
\[
b'=(a\sqcap b)'\sqcap b'=a'\sqcap b'\leq a',
\]
i.e.\ $'$ is antitone.
\end{proof}

The following remark is (almost) immediate.

\begin{remark} Let $\mb D:=(D,\sqcap,{}')$ be a commutative directoid with a unary operation $'$. Then $\mb D$ satisfies conditions \rm{(4)} and \rm{(5)} from Proposition \ref{prop1} if and only if it satisfies \[((x\sqcap y)'\sqcap y')'\approx y.\] 
\end{remark}

It is also possible to describe the lower cone $L(a,b)$ in terms of an assigned commutative directoid as follows.

\begin{lemma}\label{lem1}
Let $\mathbf P=(P,\leq)$ be a downward directed poset, $\mathbb D(\mathbf P)=(P,\sqcap)$ an assigned commutative directoid and $a,b\in P$. Then $L(a,b)=\{(x\sqcap a)\sqcap(x\sqcap b)\mid x\in P\}$.
\end{lemma}

\begin{proof}
If $c\in L(a,b)$ then $c=c\sqcap c=(c\sqcap a)\sqcap(c\sqcap b)$. Conversely, if $d\in P$ then
\begin{align*}
(d\sqcap a)\sqcap(d\sqcap b)\leq d\sqcap a\leq a, \\
(d\sqcap a)\sqcap(d\sqcap b)\leq d\sqcap b\leq b,
\end{align*}
i.e.\ $(d\sqcap a)\sqcap(d\sqcap b)\in L(a,b)$.
\end{proof}

Now we are ready to prove our first result.

\begin{theorem}\label{thm:pardirectelementclass}
Let $\mathbf P=(P,\leq,{}',0,1)$ be a bounded poset with a unary operation $'$ and $\mathbb D(\mathbf P)=(P,\sqcap,{}',0,1)$ an assigned commutative directoid. Then $\mathbf P$ is paraorthomodular if and only if $\mathbb D(\mathbf P)$ satisfies identities {\rm(4)} and {\rm(5)} and for all $x,y\in P$ the following condition:
\begin{enumerate}
\item[{\rm(6)}] $((x\sqcap y)'\sqcap z)\sqcap(x\sqcap z)=0$ for all $z\in P$ implies $x\leq y$.
\end{enumerate}
\end{theorem}

\begin{remark}
Observe that {\rm(6)} is equivalent to
\begin{enumerate}
\item[{\rm(6')}] $x\leq y$ and $(x'\sqcap z)\sqcap(y\sqcap z)=0$ for all $z\in P$ imply $x=y$.
\end{enumerate}
\end{remark}

\begin{proof}
Let $a,b,c\in P$. By Proposition~\ref{prop1}, $'$ is an antitone involution if and only if $\mathbb D(\mathbf P)$ satisfies identities (4) and (5). If $\mathbf P$ satisfies (P*) and $((a\sqcap b)'\sqcap x)\sqcap(a\sqcap x)=0$ for all $x\in P$ then, by Lemma~\ref{lem1}, $L((a\sqcap b)',a)=\{0\}$ which together with $a\sqcap b\leq a$ and (P*) implies $a=a\sqcap b\leq b$, i.e.\ (6) holds. If, conversely, (6) holds, $a\leq b$ and $L(a',b)=\{0\}$ then, by Lemma~\ref{lem1},
\[
((b\sqcap a)'\sqcap x)\sqcap(b\sqcap x)=(a'\sqcap x)\sqcap(b\sqcap x)=0
\]
for all $x\in P$ and hence $b\leq a$ which together with $a\leq b$ implies $a=b$, i.e.\ (P*) holds.
\end{proof}

If a bounded directoid $\mathbf D$ with an antitone involution is such that $\mathbb P(\mathbf D)$ is a paraorthomodular poset, then $\mathbf D$ will be called a \emph{paraorthomodular directoid}. Note that any paraorthomodular poset $\mathbf P$ can be organized into a paraorthomodular directoid. Let us denote by $\mathcal D^{\mathcal P}$ the class of {\em all} paraorthomodular directoids. Clearly, $\mathcal D^{\mathcal P}$ is assigned to $\mathcal P$. Now, since $\mathcal D^{\mathcal P}$ is an elementary class (by Theorem~\ref{thm:pardirectelementclass}), one might wonder if $\mathcal D^{\mathcal P}$ is also a universal class. The next results show that, indeed, there are sufficient conditions for answering the above question in the positive. Let $\mathcal C$ be a class of directoids. Let us denote by $\mathcal C_0$ the class of bounded involution posets induced by $\mathcal C$.

Consider a universal formula $\eta$ in the language $(\leq,0,1)$. Let us define its translation $t$ in the language $(\sqcap,0,1)$ as follows: For any literal $\phi$ that occurs in $\eta$
\begin{enumerate}
\item[(7)] $t(\phi)=
\begin{cases}
{x\sqcap y\approx x}\text{, if }\phi\text{ is }x\leq y; \\
\lnot(x\sqcap y\approx x)\text{, otherwise.}
\end{cases}
$
\end{enumerate}
Let us remark that $t$ can be obviously extended to a translation from the language $(\le,{}',0,1)$ in the language $(\sqcap,{}',0,1)$.

\begin{lemma}\label{lm:n->t(n)}
Let $\eta$ be a universal formula in the language $(\le,{}',0,1)$. Then, for any involutive directoid $\mathbf D$, $\mathbf D\models t(\eta)$ if and only if $\mathbb P(\mathbf D)\models\eta$.
\end{lemma}

\begin{proof}
For general reasons, we can safely assume that $\eta$ is in disjunctive normal form. In other words, $\eta$ corresponds to
\[
\forall\vec x\bigvee_{i\in I}\bigwedge_{j\in J}\phi_{i,j},
\]
where $\phi_{i,j}$ is either of the form $x\leq y$, or of the form $\lnot(x\leq y)$, with $x,y\in\{z,z',0,1\}$. Consider
\[
t(\forall\vec x\bigvee_{i\in I}\bigwedge_{j\in J}\phi_{i,j})=\forall\vec x\bigvee_{i\in I}\bigwedge_{j\in J}t(\phi_{i,j}).
\]
It is evident that $t(\eta)$ is also in disjunctive normal form. Let us suppose that
\[
\mathbf D\not\models t(\eta)[\vec x,\vec a]\equiv\bigvee_{i\in I}\bigwedge_{j\in J}t(\phi_{i,j})[\vec x,\vec a],
\]
with $\vec a$ elements of $D$. Then, for any $i\in I$ there is $j\in J$ such that $\mathbf D\not\models t(\phi_{i,j})[\vec x,\vec a]$. So, without loss of generality, for some $a_p,a_q\in D$, either
\begin{align*}
& a_p\sqcap a_q\neq a_p\text{ if }\phi_{i,j}\text{ is of the form }x\leq y\text{, or} \\
& a_p\sqcap a_q=a_p\text{ if }\phi_{i,j}\text{ is of the form }\lnot(x\leq y).
\end{align*}
Then, for any $i\in I$ there is $j\in J$ such that, $\mathbb P(\mathbf D)\not\models\phi_{i,j}$. The converse direction is analogous.
\end{proof}

The next result shows that if a class $\mathcal C_0$ of bounded posets with antitone involution induced by a class $\mathcal C$ of bounded directoids with antitone involution is universal, then $\mathcal C$ is universal.

\begin{corollary}\label{lm:C->C_0}
If $\mathcal C_0$ is a universal class, then $\mathcal C$ is universal.
\end{corollary}

\begin{proof}
Let us assume that $\mathcal C_0$ is a universal class. Then there exists a family of universal formulae $\{\eta_i\}_{i\in I}$ such that $\mathbf P\in\mathcal C_0$ if and only if $\mathbf P\models\eta_i$, for any $i\in I$. Now, for any $i\in I$, let $t(\eta_i)$ be the formula obtained by applying the translation in condition (7). We claim that $\{t(\eta_i)\}_{i\in I}$ defines $\mathcal C$. Consider $\mathbf D\in\mathcal C$. By assumption, for all $i\in I$, $\mathbb P(\mathbf D)\models\eta_i$. Then, by Lemma~\ref{lm:n->t(n)}, we have that $\mathbf D\models t(\eta_i)$. Conversely, if $\mathbf D\models t(\eta_i)$ (for any $i\in I$), then $\mathbb P(\mathbf D)\models\eta_i$, for any $i\in I$, again by Lemma~\ref{lm:n->t(n)}. Hence, $\mathbb P(\mathbf D)\in\mathcal C_0$ and, since $\mathcal C$ is assigned to $\mathcal C_0$, one has $\mathbf D\in\mathcal C$.
\end{proof}

Note that, in general, the converse of Corollary~\ref{lm:C->C_0} need not hold. In fact, paraorthomodular lattices form a quasivariety of directoids. However, it is easily seen that the class of their $\{\leq,{}',0,1\}$-reducts is not closed under substructures.

Indeed, we have the following:

\begin{theorem}\label{prop:pardirectnotunivclass}
The class $\mathcal D^{\mathcal P}$ does not form a universal class.
\end{theorem}

\begin{proof}
Consider the paraorthomodular poset (lattice) $\mathbf P=(P,\leq,{}',0,1)$ depicted in Fig.~5:

\vspace*{-2mm}

\begin{center}
\setlength{\unitlength}{7mm}
\begin{picture}(4,12)
\put(2,1){\circle*{.3}}
\put(2,3){\circle*{.3}}
\put(1,5){\circle*{.3}}
\put(3,5){\circle*{.3}}
\put(1,7){\circle*{.3}}
\put(3,7){\circle*{.3}}
\put(2,9){\circle*{.3}}
\put(2,11){\circle*{.3}}
\put(2,1){\line(0,1)2}
\put(2,3){\line(-1,2)1}
\put(2,3){\line(1,2)1}
\put(1,5){\line(0,1)2}
\put(3,5){\line(0,1)2}
\put(2,9){\line(-1,-2)1}
\put(2,9){\line(1,-2)1}
\put(2,9){\line(0,1)2}
\put(1.85,.25){$0$}
\put(2.4,2.85){$c'$}
\put(.3,4.85){$b'$}
\put(3.4,4.85){$a$}
\put(.3,6.85){$a'$}
\put(3.4,6.85){$b$}
\put(2.4,8.85){$c$}
\put(1.3,11.4){$1=0'$}
\put(1.2,-.75){{\rm Fig.\ 5}}
\end{picture}
\end{center}

\vspace*{4mm}

%\begin{figure}
%\begin{equation}
%  \xymatrix{
%&&1\ar@{-}[d]&&\\
%&&c'\ar@{-}[dl]\ar@{-}[dr]&&\\
%&a'\ar@{-}[d]&&b\ar@{-}[d]&\\
%&b'\ar@{-}[dr]&&a\ar@{-}[dl]&\\
%&&c\ar@{-}[d]&&\\
%&&0&&
%}\tag{}
%\end{equation}
%\caption{the paraorthomodular lattice $\mathbf {F}_{8}$}\label{fig:parab6}
%\end{figure}

Now, consider the involutive directoid $\mathbb D(\mathbf P)$ assigned to $\mathbf {P}$ such that
\[
b'\sqcap a=a\sqcap b'=a'\sqcap b=b\sqcap a'=0,
\]
and $x\sqcap y=\min\{x,y\}$ for any other pair of elements. Of course, $\mathbb D(\mathbf P)$ is a paraorthomodular directoid. However, it can be easily seen that the sub-involutive directoid of $\mathbb D(\mathbf P)$ generated by $a,b$ is not. Hence, the class $\mathcal D^\mathcal P$ is not closed under substructures and therefore it is not a universal class.
\end{proof}

Actually, Theorem~\ref{prop:pardirectnotunivclass} suggests that, in order to assign (at least) a quasivariety of directoids to the class $\mathcal P$ of paraorthomodular posets, a stricter policy by means of which the operation $\sqcap$ is defined must be adopted. In other words, the class $\mathcal D^\mathcal P$ is ``too large''.

Let $\mathbf P=(P,\leq,{}',0,1)$ be a paraorthomodular poset. Define, for any $x,y\in P$,
\[
x\sqcap y:=\left\{
\begin{array}{ll}
\min\{x,y\}                                         & \text{if }x\not\parallel y, \\
\text{an arbitrary element of }L(x,y)\setminus\{0\} & \text{if }x'<y\text{ and }L(x,y)\neq\{0\}, \\
\text{an arbitrary element of }L(x,y)               & \text{otherwise}.
\end{array}
\right.
\]
Clearly, $\mathbb D(\mathbf P)=(P,\sqcap,{}',0,1)$ is a bounded directoid with an antitone involution. A bounded directoid $\mathbf D$ with an antitone involution obtained from a bounded poset $\mathbf P$ with an antitone involution by means of the above procedure will be said to be {\em canonically} assigned to $\mathbf P$. If a bounded directoid $\mathbf D$ with an antitone involution is canonically assigned to a bounded poset $\mathbf P$ with an antitone involution, we will write $\mathbf D=\mathbb D^c(\mathbf P)$. Given a class $\mathcal C$ of bounded posets with an antitone involution, and a class $\mathcal D$ of bounded directoids with an antitone involution, we will say that $\mathcal D$ is canonically assigned to $\mathcal C$ if
\begin{align*}
\mathcal C & =\{\mathbf P\mid\text{there exists some }\mathbf D\in\mathcal D\text{ with }\mathbb P(\mathbf D)=\mathbf P\}\text{ and} \\
\mathcal D & =\{\mathbf D\mid\text{there exists some }\mathbf P\in\mathcal C\text{ with }\mathbb D^c(\mathbf P)=\mathbf D\}.
\end{align*}
From now on, we will denote by $\mathcal D_c^\mathcal P$ the class of bounded directoids with an antitone involution canonically assigned to $\mathcal P$.

\begin{proposition}\label{prop:dpc form a quasivar}
$\mathcal D_c^\mathcal P$ forms a quasivariety axiomatized by {\rm(4)} and {\rm(5)} of Proposition~\ref{prop1}, and the following quasi-identity
\begin{enumerate}
\item[{\rm(8)}] $(x\sqcap y)'\sqcap y=0$ implies $x\sqcap y=y$.
\end{enumerate}
\end{proposition}

\begin{proof}
Suppose that $\mathbf D=\mathbb D^c(\mathbf P)$, for some $\mathbf P\in\mathcal P$. If $(x\sqcap y)'\sqcap y=0$, since $(x\sqcap y)''=x\sqcap y\leq y$, then one must have $L((x\sqcap y)',y)=\{0\}$. Since $\mathbf P$ is paraorthomodular, we have $x\sqcap y=y$. Conversely, suppose that $\mathbf D$ is a bounded directoid with an antitone involution satisfying (8). Clearly $\mathbb P(\mathbf D)$ is a bounded lattice with an antitone involution. Moreover, if $x\leq y$ and $L(x',y)=\{0\}$, then $x'\sqcap y=(x\sqcap y)'\sqcap y=0$ and, by (8), $x\sqcap y=y$, i.e.\ $x=y$. Therefore $\mathbb P(\mb D)$ is paraorthomodular. Let us show $\mathbf D=\mathbb D^c(\mathbb P(\mathbf D))$. To this aim, just note that, assuming by way of contradiction that $x'<y$, $L(x,y)\neq\{0\}$ and $x\sqcap y=0$ one has, by (8), $x'=y$.
\end{proof}

\begin{proposition}\label{prop: dcanonical not var}
$\mathcal D_c^\mathcal P$ does not form a variety.
\end{proposition}

\begin{proof}
If $\mathcal D_c^\mathcal P$ forms a variety, then there exists a (possibly infinite) set of equations $\mathcal E$ such that $\mathcal D_c^\mathcal P=\{\mathbf D\mid\mathbf D\models\mathcal E\}$. Now, consider the set of equations $\mathcal E':=\mathcal E\cup\{\text{(ass),(reg)}\}$, where
\begin{enumerate}
\item[(ass)] $x\sqcap(y\sqcap z)\approx(x\sqcap y)\sqcap z$,
\item[(reg)] $(x\sqcap x')\sqcap(y\sqcup y')\approx x\sqcap x'$.
\end{enumerate}
Clearly, the class $\mathcal L=\{\mathbf D\mid\mathbf D\models\mathcal E'\}$ of bounded directoids with an antitone involution is again a variety. Now, let us denote by $\mathbb{PL}$ the class (quasivariety) of paraorthomodular lattices (in the sense of \cite{GLP16}). We show that $\mathbb{PL}$ coincides with $\mathcal L$. Firstly, note that, for any $\mathbf L\in\mathbb{PL}$, setting $x\land y=x\sqcap y$, for any $x,y\in L$, $\mathbf L$ is a bounded directoid with an antitone involution satisfying (8), (ass) and (reg). Hence, by hypothesis it must satisfy $\mathcal E'$. Therefore, we conclude $\mathbb{PL}\subseteq\mathcal L$. Conversely, since $\mathcal L$ satisfies (ass), by \cite{CL}, any member $\mathbf D$ of $\mathcal L$ is a meet-semilattice and, since $'$ is an antitone involution, it is also a lattice, by De Morgan's laws. Moreover, since $\mathbf D$ satisfies (8) and (reg), it is a paraorthomodular lattice. Therefore, we conclude that $\mathbb{PL}=\mathcal L$. However, by \cite[Theorem~3.2]{GLP16}, $\mathbb{PL}$ is not closed under homomorphic images and therefore it does not form variety, a contradiction. We conclude that $\mathcal D_c^\mathcal P$ does not form a variety.
\end{proof}

\begin{proposition}
The variety generated by $\mathcal D^\mathcal P_c$ coincides with the class of bounded directoids with an antitone involution.
\end{proposition}

\begin{proof}
In order to prove the statement, let us show that any bounded directoid $\mathbf D$ with an antitone involution is a quotient of some $\mathbf M\in\mathcal D^\mathcal P_c$ by mimicking the same technique exploited in the proof of \cite[Theorem~3.2]{GLP16}. So, let $\mathbf D$ be a bounded directoid with an antitone involution and $a,a'\notin D$. Furthermore, define
\[
\mathbf M:=(D\cup\{a,a'\},\sqcap^\mathbf M,{}{'}^\mathbf M,0,1)
\]
such that:
\begin{equation*}
x{'}^\mathbf M=
\begin{cases}
x{'}^\mathbf D & \text{if }x\in D, \\
a'             & \text{if }x=a, \\
a              & \text{if }x=a',
\end{cases}
\end{equation*}
\begin{equation*}
x\sqcap^\mathbf My=
\begin{cases}
0                  & \text{if }0\in\{x,y\}, \\
x\sqcap^\mathbf Dy & \text{if }\{x,y\}\subseteq D\text{ and }x\sqcap^\mathbf Dy\neq0, \\
a                  & \text{if }\{x,y\}\subseteq D\setminus\{0\},x\sqcap^\mathbf Dy=0, \\
a                  & \text{if }a\in\{x,y\}\text{ and }0\notin\{x,y\}, \\
x(y)               & \text{if }a'=y\ (a'=x),\text{and }1\notin\{x,y\}.
\end{cases}
\end{equation*}
Obviously, $(D\cup\{a,a'\},\sqcap^\mathbf M,{}{'}^\mathbf M,0,1)$ is a bounded directoid with an antitone involution. Moreover, $\mathbf M$ satisfies (8). In fact, observe that, for any $x,y\in M$, $x\sqcap y=0$ if and only if $x=0$ or $y=0$. Hence, if $(x\sqcap y)'\sqcap y=0$, then $(x\sqcap y)'=0$ or $y=0$. In the first case, $(x\sqcap y)'=0$ entails $x\sqcap y=1$, since $'$ is an antitone involution. Therefore, we have $x=y=1$ and so $x\sqcap y=1\sqcap1=1=y$. In the second case our conclusion follows trivially upon recalling that $0$ is the least element. We conclude that $\mathbf M\in\mathcal D^\mathcal P_c$. Finally, consider the equivalence relation $\Theta$ defined as follows:
\begin{align*}
[0]\Theta & :=\{0,a\}, \\
[1]\Theta & :=\{1,a'\}, \\
[x]\Theta & :=\{x\}\text{ for }a<x<a'.
\end{align*}
Clearly, $\Theta$ is a congruence. Moreover, it is evident that $\mathbf M/\Theta\cong\mathbf D$.
\end{proof}

There arises the natural question whether there exists a largest quasivariety of paraorthomodular directoids. Namely, one might ask if there exists a quasivariety $\mathcal C$ of paraorthomodular directoids such that, if $\mathcal C_1$ is any other quasivariety of bounded directoids with an antitone involution whose induced bounded posets with an antitone involution are paraorthomodular, then $\mathcal C_1\subseteq\mathcal C$. The next proposition solves the problem in the positive.

\begin{proposition}
\
\begin{enumerate}[{\rm(i)}]
\item The largest quasivariety $\mathcal{PD}$ of paraorthomodular directoids assigned to $\mathcal P$ is axiomatized by $(4)$ and $(5)$ of Proposition~\ref{prop1} and the following quasi-identity:
\begin{enumerate}
\item[{\rm(9)}] $((x\sqcap y)'\sqcap y)\sqcup(y'\sqcap(x\sqcap y))=0$ implies $x\sqcap y=y$.
\end{enumerate}
\item $\mathcal{PD}$ does not form a variety.
\end{enumerate}
\end{proposition}

\begin{proof}
\
\begin{enumerate}[(i)]
\item Verifying that, for any $\mathbf D\in\mathcal{PD}$, $\mathbb P(\mathbf D)$ is a paraorthomodular poset is straightforward and therefore it is left to the reader. Now let $\mathbf P$ be a paraorthomodular poset. Let us consider $\mathbb D^c(\mathbf P)$. By Proposition~\ref{prop:dpc form a quasivar}, $\mathbb D^c(\mathbf P)$ satisfies (8) and so one can easily see that it satisfies (9), too. We conclude $\mathbb D^c(\mathbf P)\in\mathcal{PD}$ and so $\mathcal{PD}$ is assigned to $\mathcal P$. Finally, assume by way of contradiction that $\mathcal D$ is a quasi-variety of paraorthomodular directoids such that, for some $\mathbf D\in\mathcal D$, $\mathbf D$ does not satisfy (9). This means that there exist $a,b\in D$ such that
\[
((a\sqcap b)'\sqcap b)\sqcup(b'\sqcap(a\sqcap b))=0,
\]
but $a\sqcap b<a$. It is easily noticed that the bounded poset with involution induced by the involutive subdirectoid of $\mathbf D$ generated by $\{a\sqcap b,b\}$ is ortho-isomorphic to $\mathbf B_6$. Hence, $\mathcal D$ is not closed under substructures and thus it does not form a quasivariety contradicting our assumptions. We conclude that $\mathcal D\subseteq\mathcal{PD}$.
\item This claim follows by applying the same argument from Proposition~\ref{prop: dcanonical not var}.
\end{enumerate}
\end{proof}

In what follows we present another and stronger generalization of paraorthomodular lattices.

\begin{definition}\label{def1}
A {\em sharply paraorthomodular poset} is a bounded poset $(P,\leq,{}',0,1)$ with an antitone involution $'$ satisfying for all $x,y\in P$ the conditions
\begin{enumerate}
\item[{\rm(10)}] $x\leq y'$ implies that $x\vee y$ exists,
\item[{\rm(11)}] $x\leq y$ and $x'\wedge y=0$ imply $x=y$.
\end{enumerate}
\end{definition}

Let us note that (10) and (11) coincide with (2) and (P), respectively.

Of course, every sharply paraorthomodular poset is paraorthomodular.

It is easy to see that in every poset $(P,\leq,{}')$ with an antitone involution the following conditions are equivalent:
\begin{enumerate}
\item[(12)] $x\leq y'$ implies that $x\vee y$ exists,
\item[(13)] $x\leq y$ implies that $x'\wedge y$ exists.
\end{enumerate}
Examples of sharply paraorthomodular posets abound in quantum structures literature. In fact, any orthomodular poset and any $\{\leq,',0,1\}$-reduct $\mc P(\alga)$ of a lattice effect algebra $\mb A$ are obviously sharply paraorthomodular, by virtue of Lemma \ref{lem: ompparaorthomod} and Proposition \ref{prop}, respectively. Furthermore, any $\{\leq,',0,1\}$-reduct $\mc P(\alga)$ of a $4$-loop pasting $\mb A$ (of an admissible system $\mc F$) of MV-algebras \cite[Definition 12]{FCHO} is an example of a proper sharply paraorthomodular poset. Indeed, by \cite[pp. 356--357]{FCHO}, $\mb A$ is an effect algebra. So, by Proposition \ref{prop}, $\mc P(\alga)$ is paraorthomodular. Moreover, by \cite[Theorem 16]{FCHO}, although $\mb A$ is not lattice-ordered, joins of orthogonal elements always exist in $A$. Hence, $\mc P(\alga)$ is  sharply paraorthomodular.

We are going to characterize sharply paraorthomodular posets by means of assigned commutative directoids.

\begin{lemma}\label{lem3}
Let $\mathbf P=(P,\leq,{}',0,1)$ be a bounded poset with an antitone involution $'$ and $\mathbb D(\mathbf P)=(P,\sqcap,{}',0,1)$ an assigned commutative directoid. Then $\mathbf P$ satisfies {\rm(10)} if and only if $\mathbb D(\mathbf P)$ satisfies the identity
\begin{enumerate}
\item[{\rm(14)}] $z\leq((x\sqcup z)\sqcup(y\sqcup z)')\sqcap(y\sqcup z)$.
\end{enumerate}
\end{lemma}

\begin{proof}
Let $a,b,c\in P$. First assume that $\mathbf P$ to satisfy condition (10) which is equivalent to (13). Put
\begin{align*}
d & :=(a\sqcup c)\sqcup(b\sqcup c)', \\
e & :=b\sqcup c.
\end{align*}
Then $e'\leq d$, i.e.\ $d'\leq e$ and hence by (13) $d\wedge e=d''\wedge e$ exists, thus $d\sqcap e=d\wedge e$, and because of $c\leq d,e$ we have
\[
c\leq d\wedge e=d\sqcap e=((a\sqcup c)\sqcup(b\sqcup c)')\sqcap(b\sqcup c),
\]
i.e.\ (14) holds. Conversely, assume (14) holds. If $a\leq b$ and $c\leq a',b$ then
\[
c\leq((a'\sqcup c)\sqcup(b\sqcup c)')\sqcap(b\sqcup c)=(a'\sqcup b')\sqcap b=a'\sqcap b
\]
which means $a'\sqcap b=a'\wedge b$ and hence $a'\wedge b$ exists, i.e.\ (13) and therefore also (10) holds.
\end{proof}

Now, we characterize sharply paraorthomodular posets by an identity and a quasi-identity of an assigned commutative directoid as follows.

\begin{theorem}\label{thm: sharply paraorthdir}
Let $\mathbf P=(P,\leq,{}',0,1)$ be a bounded poset with an antitone involution $'$ and $\mathbb D(\mathbf P)=(P,\sqcap,{}',0,1)$ an assigned commutative directoid. Then $\mathbf P$ is sharply para\-ortho\-mod\-u\-lar if and only if $\mathbb D(\mathbf P)$ satisfies identity {\rm(14)} and quasi-identity {\rm(8)}.
\end{theorem}

\begin{proof}
Let $a,b\in P$. First assume $\mathbf P$ to be sharply paraorthomodular. Then, because of Lemma~\ref{lem3}, (14) holds. If $(a\sqcap b)'\sqcap b=0$ then, since $a\sqcap b\leq b$ and (13) is equivalent to (10), we have that $(a\sqcap b)'\wedge b$ exists and hence $(a\sqcap b)'\wedge b=(a\sqcap b)'\sqcap b=0$ whence $a\sqcap b=b$ according to (11), i.e.\ (8) holds. Conversely, assume $\mathbb D(\mathbf P)$ satisfies (14) and (8). Then, because of Lemma~\ref{lem3}, (10) holds. If $a\leq b$ and $a'\wedge b=0$ then $(a\sqcap b)'\sqcap b=a'\sqcap b=a'\wedge b=0$ whence by (8) $a=a\sqcap b=b$, i.e.\ (11) holds.
\end{proof}

\begin{corollary}
The class of commutative directoids assigned to the class of sharply paraorthomodular posets forms a quasivariety of algebras.
\end{corollary}

In what follows, we will call any member from the quasivariety of commutative directoids assigned to sharply paraorthomodular posets a \emph{paraorthomodular directoid}.

\subsection{Sharp quantum structures as subvarieties of $\mathcal D^\mathcal P_c$}
As it has been recalled in Section \ref{sec:1}, given an effect algebra $\mb A=(A,\oplus,0,1)$ and $x,y\in P$, $x\oplus y$ need not exist provided that $x\not\leq y'$. Therefore, $\mb A$ is a partial algebra. In \cite{CKH}, the first author, R. Hala\v s and J. K\"uhr  show that any effect algebra $\mb A$ can be made into a total algebra by means of organizing its underlying induced poset into a commutative directoid and then term-defining a new total binary operation $\overline{\oplus}$ in terms of $\sqcap$ and $\oplus$. Furthermore, if $\mb A$ is an orthomodular poset, then one can do more. In fact, any orthomodular poset $\mb P$ can be arranged into a total algebra by assigning a bounded commutative directoid $\mathbb{D}(\mb P)$ with antitone involution to $\mb P$ satisfying further equational conditions.

\begin{definition}[cf.\ \cite{CK}]\label{def: orthomodulardirectoid} An \emph{orthomodular} directoid is a bounded commutative directoid with antitone involution $\mb D = (D,\sqcap,{}',0,1)$ satisfying \begin{enumerate}
\item $x\sqcup x'\approx 1$,
\item $(((x\sqcup z)\sqcup(y\sqcup z)')\sqcap(y\sqcup z))\sqcap z\approx z$,
\item $(x\sqcap y)\sqcup((x\sqcap y)\sqcup y')'\approx y$,
\end{enumerate}
where $x\sqcup y = (x'\sqcap y')'$.
\end{definition}
Since a directoid $\mb D=(D,\sqcap,{}')$ with antitone involution is a lattice if and only if $\sqcap$ is associative (see \cite{CL}), it is easily seen that the variety of orthomodular lattices $\mc{OML}$ coincides with the variety of associative orthomodular directoids.
 
Before stating and proving the main results of this section, with the next proposition we clarify the relationship between orthomodular and sharply paraorthomodular directoids (posets). 
\begin{proposition} A sharply paraorthomodular directoid $\mb D$ is orthomodular if and only if it satisfies $1.$ of Definition \ref{def: orthomodulardirectoid}.
\end{proposition}
\begin{proof}
Concerning the non-trivial direction, just note that $(x\sqcap y)\sqcup((x\sqcap y)\sqcup y')'=(x\sqcap y)\lor((x\sqcap y)\sqcup y')'\leq y$, since $(x\sqcap y)\leq ((x\sqcap y)\sqcup y')'' =((x\sqcap y)\sqcup y')$,  $'$ is an antitone involution and by Theorem \ref{thm: sharply paraorthdir}. Hence $(x\sqcap y)\sqcup((x\sqcap y)\sqcup y')'\sqcap y=(x\sqcap y)\sqcup((x\sqcap y)\sqcup y')'$. Moreover, it is easily seen that 
\begin{align*}
 [(x\sqcap y)\sqcup((x\sqcap y)\sqcup y')']'\sqcap y=& [(x\sqcap y)'\sqcap ((x\sqcap y)\sqcup y')]\sqcap y\\
 =&  [(x\sqcap y)'\land ((x\sqcap y)\sqcup y')]\land y\\
 =& ((x\sqcap y)'\land y)\land((x\sqcap y)\sqcup y')\\
 =& ((x\sqcap y)\sqcup y')'\sqcap((x\sqcap y)\sqcup y') = 0,
\end{align*}
by $1.$ upon noticing that $x\sqcap x' = (x\sqcup x')' = 1' = 0$. Therefore, by Theorem \ref{thm: sharply paraorthdir}, one has $(x\sqcap y)\sqcup((x\sqcap y)\sqcup y')'=(x\sqcap y)\sqcup((x\sqcap y)\sqcup y')'\sqcap y = y$.
\end{proof}
The next results are easy adaptations of Theorem 2 and Theorem 3 from \cite{CK}. Firstly, we have the following

\begin{theorem}\label{orthomodulardir->orthomodularpos} Let $\mb D=(D,\sqcap,{}',0,1)$ be an orthomodular directoid and $\leq$ be its induced order. Then $\mathbb{P}(\mb D)=(D,\leq,{}',0,1)$ is an orthomodular poset where for orthogonal elements $x,y\in D$ we have \[x\sqcup y = x\lor y.\]
\end{theorem}
Conversely, one has also
\begin{theorem}\label{orthomodularpos->orthomodulardir} Let $\mb P=(P,\leq,{}',0,1)$ be an orthomodular poset. Define a binary operation $\sqcup$ on $P$ as follows:
\begin{equation*}
y\sqcup x=x\sqcup y:=
\begin{cases}
x\lor y         & \text{if } x\lor y\ \text{exists}, \\
z\in U(x,y) & \text{otherwise}.
\end{cases}
\end{equation*}
Then $\mathbb{D}(\mb P)=(P,\sqcap,',0,1)$, where $x\sqcap y=(x'\sqcup y')'$, is an orthomodular directoid.
\end{theorem}
By Theorem \ref{orthomodularpos->orthomodulardir}, to every orthomodular poset $\mb P=(P,\leq,{}',0,1)$ there can be assigned a total algebra $\mathbb{D}(\mb P)$ which is an orthomodular directoid. By Theorem \ref{orthomodulardir->orthomodularpos}, one can assign an orthomodular poset $\mathbb{P}(\mathbb{D}(\mb P))$ to $\mathbb{D}(\mb P)$. Since the underlying bounded poset with antitone involution $(P,\leq,{}',0,1)$ coincides in $\mb P$, $\mathbb{D}(\mb P)$ and $\mathbb{P}(\mathbb{D}(\mb P))$, we conclude that $\mb P= \mathbb{P}(\mathbb{D}(\mb P))$. However, given an orthomodular directoid $\mb D$, then $\mb D=\mathbb{D}(\mathbb{P}(\mb D))$ need not hold. In fact, by the definition of $\sqcup$, $x\sqcup y$ can be defined in many different ways provided that $x\lor y$ does not exist.

Recall that an {\em orthoalgebra} is an effect algebra the corresponding poset of which is an orthoposet. It is well known that orthoalgebras generalize orthomodular posets with respect to effect algebras as follows. If $\mb A=(A,\oplus,',0,1)$ is an orthoalgebra, then for any $x,y\in A$ such that $x\leq y'$, $x\oplus y$ is a least upper bound of $x$ and $y$ in $A$ although it need not be their supremum. Naturally the question 
arises whether Theorem \ref{orthomodulardir->orthomodularpos} and Theorem \ref{orthomodularpos->orthomodulardir} can be generalized to orthoalgebras, with the aim of proving that any orthoalgebra $\mb A$ can be made into a total algebra by suitably assigning a bounded commutative directoid $\mathbb{D}(\mb A)$ with antitone involution  to $\mb A$. The next results solve the aforementioned problem in the positive.
\begin{definition}\label{def: odir}
An {\em ortho-directoid} is an algebra $(D,\sqcap,{}',0,1)$ of type $(2,1,0,0)$ satisfying the following identities:
\begin{enumerate}[{\rm(i)}]
\item $x\sqcap y\approx y\sqcap x$,
\item $x\sqcup x'\approx 1$,
\item $0\sqcup x\approx x$,
\item $((x\sqcap y')\sqcup y)\sqcup(((x\sqcap y')\sqcup y)\sqcup z)'\approx(x\sqcap y')\sqcup (y\sqcup(((x\sqcap y')\sqcup y)\sqcup z)')$,
\item $(x\sqcap y)\sqcup((x\sqcap y)\sqcup y')'\approx y$,
\item $(x\sqcap y')\sqcup(y\sqcup(((x\sqcap y')\sqcup y)\sqcup z)')'\approx(y\sqcup(((x\sqcap y')\sqcup y)\sqcup z)')'$
\end{enumerate}
where $x\sqcup y:=(x'\sqcap y')'$ for all $x,y\in D$.
\end{definition}

Note that, by the definition of $\sqcup$ in any ortho-directoid we have $x\sqcup y\approx y\sqcup x$. Let us denote by $\mc{OD}$ the variety of ortho-directoids.

In order to prove the next lemma, we will freely make use of axioms from Definition~\ref{def: odir}.

\begin{lemma}\label{lem: auxo-direct}
Let $\mathbf D=(D,\sqcap,{}',0,1)$ be an ortho-directoid. Then {\rm(a)} and {\rm(b)} hold:
\begin{enumerate}[{\rm(a)}]
\item $\mathbf D$ satisfies the following identities and quasi-identity:
\begin{enumerate}[{\rm(i)}]
\item $0'\approx1$,
\item $x''\approx x$,
\item $1'\approx0$,
\item $x\sqcap x'\approx x\sqcap 0\approx0$, $x\sqcap1\approx x$, $(x\sqcap y)\sqcup y\approx(x\sqcup y)\sqcap y\approx y$, $x\sqcup1\approx1$,
\item $(x\sqcap y)'\sqcap y'\approx y'$,
\item $(x\sqcap y)'\sqcap y=0\text{ implies }x\sqcap y=y$,
\item $x\sqcap x\approx x$,
\item $x\sqcap((x\sqcap y)\sqcap z)\approx((x\sqcap y)\sqcap z)$.
\end{enumerate}
\item $\mathbb P(\mathbf D)$ is an orthoposet.
\end{enumerate}
\end{lemma}

\begin{proof}
\
\begin{enumerate}[(a)]
\item
\begin{enumerate}[(i)]
\item $0'\approx0\sqcup 0'\approx1$.
\item Observe that, by the definition of $\sqcup$, one has
\[
1\approx x\sqcup x'\approx(x'\sqcap x'')'
\]
and, since
\[
(x\sqcap y)\sqcup((x\sqcap y)\sqcup x')'\approx x,
\]
we have
\[
x\approx((x\sqcap y)'\sqcap((x\sqcap y)'\sqcap x'')''')'.
\]
Hence, we compute
\begin{align*}
x' & \approx((x'\sqcap x'')'\sqcap ((x'\sqcap x'')'\sqcap x''')''')'\approx((1\sqcap((1\sqcap x''')''')'\approx \\
   & \approx((0'\sqcap((0'\sqcap x''')''')'\approx0\sqcup(0'\sqcap x''')''\approx(0'\sqcap x''')''\approx(0\sqcup x'')'\approx x'''.
\end{align*}
So, we have
\[
x\approx((x\sqcap x'')'\sqcap((x\sqcap x'')'\sqcap x'')''')'\approx((x''\sqcap x)'\sqcap((x''\sqcap x)'\sqcap x'''')''')'\approx x''.
\]
\item Just note that $1'\approx0''\approx0$.
\item We have
\[
(x\sqcap x')'\approx(x''\sqcap x')'\approx x'\sqcup x\approx1.
\]
Therefore,
\[
x\sqcap x'\approx(x\sqcap x')''\approx1'\approx0.
\]
Furthermore, observe that
\[
x\approx x\sqcup0\approx(x'\sqcap0')'\approx(x'\sqcap1)'.
\]
By applying $'$ to both sides of the previous identity, one has that $x'\approx x'\sqcap1$ and so $x\approx x\sqcap1$, since $'$ is an involution. Now, note that
\[
(((x\sqcap y')\sqcup y)\sqcup((x\sqcap y')\sqcup y)')'\approx0.
\]
Therefore, one has 
\begin{align*}
(x\sqcap y')\sqcup y' & \approx(x\sqcap y')\sqcup(y\sqcup 0)'\approx \\
                      & \approx(x\sqcap y')\sqcup(y\sqcup(((x\sqcap y')\sqcup y)\sqcup((x\sqcap y')\sqcup y)')')'\approx \\
                      & \approx(y\sqcup(((x\sqcap y')\sqcup y)\sqcup((x\sqcap y')\sqcup y)')')'\approx y'.
\end{align*}
Hence, we conclude $(x\sqcap y)\sqcup y\approx y$. The dual form of the obtained equation can be derived by definition of $\sqcup$ and the properties of $'$. However, we have also
\[
x\sqcap0\approx(x\sqcap0)\sqcup0\approx0
\]
as well as $x\sqcup1\approx1$, by De Morgan's laws.
\item Note that
\[
((x\sqcap y)'\sqcap y')'\approx(x\sqcap y)\sqcup y\approx y.
\]
Therefore we obtain $(x\sqcap y)'\sqcap y'\approx y'$ since $'$ is an involution.
\item This easily follows upon noticing that
\[
y\approx(x\sqcap y)\sqcup((x\sqcap y)\sqcup y')'\approx(x\sqcap y)\sqcup((x\sqcap y)'\sqcap y).
\]
\item We have
\[
x\sqcap x\approx(x\sqcup(x\sqcap x))\sqcap x\approx x.
\]
\item Observe that 
\begin{align*}
(x\sqcap(x\sqcap(x\sqcap y)')')' & \approx x'\sqcup(x'\sqcup(x\sqcap y)'')'\approx \\
                                 & \approx((x\sqcap y)'\sqcap x')\sqcup(((x\sqcap y)'\sqcap x')\sqcup(x\sqcap y)'')'\approx(x\sqcap y)'.
\end{align*}
The above derivation proves that 
\begin{enumerate}
\item[(15)] $x\sqcap(x\sqcap(x\sqcap y)')'\approx x\sqcap y$.
\end{enumerate}
Furthermore, we have
\begin{enumerate}
\item[(16)] $(x\sqcap y)'\sqcap(y\sqcap(x\sqcap y)')'\approx y'$
\end{enumerate}
since
\[
y'\approx((x\sqcap y)\sqcup((x\sqcap y)\sqcup y')')'.
\]
Now, by Definition~\ref{def: odir}, one has
\[
((x\sqcap y')\sqcup(y\sqcup(((x\sqcap y')\sqcup y)\sqcup z)'))'\approx(((x\sqcap y')\sqcup y)\sqcup(((x\sqcap y')\sqcup y)\sqcup z)')'.
\]
Note that
\begin{align*}
(((x\sqcap y')\sqcup y)\sqcup(((x\sqcap y')\sqcup y)\sqcup z)')' & \approx((x\sqcap y')'\sqcap y')\sqcap(((x\sqcap y')\sqcup y)'\sqcup z)\approx \\
                                                                 & \approx((x\sqcap y')'\sqcap y')\sqcap(((x\sqcap y')'\sqcap y')\sqcap z')'
\end{align*}
as well as
\[
((x\sqcap y')\sqcup(y\sqcup(((x\sqcap y')\sqcup y)\sqcup z)'))'\approx(x\sqcap y')'\sqcap(y'\sqcap(((x\sqcap y')'\sqcap y')'\sqcap z')').
\]
Hence, we derive
\begin{enumerate}
\item[(17)] $((x\sqcap y')'\sqcap y')\sqcap(((x\sqcap y')'\sqcap y')\sqcap z')'\approx(x\sqcap y')'\sqcap(y'\sqcap(((x\sqcap y')'\sqcap y')'\sqcap z')')$.
\end{enumerate}
Now, let us set $y:=x\sqcap(y\sqcap x)'$ in (17). Then we obtain
\begin{align*}
& ((x\sqcap(x\sqcap(y\sqcap x)')')'\sqcap(x\sqcap(y\sqcap x)')')\sqcap \\
& \sqcap(((x\sqcap(x\sqcap(y\sqcap x)')')'\sqcap(x\sqcap(y\sqcap x)'))\sqcap z')'\approx \\
& \approx(x\sqcap(x\sqcap(y\sqcap x)')')\sqcap \\
& \sqcap((x\sqcap(y\sqcap x)')'\sqcap(((x\sqcap(x\sqcap(y\sqcap x)')')'\sqcap(x\sqcap(y\sqcap x)')')'\sqcap z')').
\end{align*}
By (15) one has
\begin{align*}
& (((y\sqcap x)'\sqcap(x\sqcap(y\sqcap x)')')\sqcap(((y\sqcap x)'\sqcap(x\sqcap(y\sqcap x)')')\sqcap z')'\approx \\
& \approx(y\sqcap x)'\sqcap((x\sqcap(y\sqcap x)')'\sqcap(((y\sqcap x)'\sqcap(x\sqcap(y\sqcap x)')')'\sqcap z')'.
\end{align*}
By (16), the last identity can be rewritten as
\[
x'\sqcap(x'\sqcap z')'\approx(y\sqcap x)'\sqcap((x\sqcap(y\sqcap x)')'\sqcap(x\sqcap z')').
\]
Hence, we have
\[
(x'\sqcap(x'\sqcap z')')'\approx((y\sqcap x)'\sqcap((x\sqcap(y\sqcap x)')'\sqcap(x\sqcap z')'))'.
\]
The last identity entails
\[
(x\sqcap y)\sqcap(x'\sqcap (x'\sqcap z')')'\approx(x\sqcap y)\sqcap((y\sqcap x)'\sqcap((x\sqcap(y\sqcap x)')'\sqcap(x\sqcap z')'))'.
\]
Therefore, by (v), one has
\[
(x\sqcap y)\sqcap(x'\sqcap(x'\sqcap z')')'\approx x\sqcap y.
\]
Now, replacing $z$ by $x'\sqcap z'$, we have
\[
(y\sqcap x)\sqcap(x'\sqcap(x'\sqcap(x'\sqcap z')')')'\approx y\sqcap x,
\]
namely
\[
(y\sqcap x)\sqcap(x'\sqcap z')'\approx x\sqcap y,
\]
by (15). Thus, by definition of $\sqcup$, we have
\[
(y\sqcap x)\sqcap(x\sqcup z)\approx x\sqcap y.
\]
Finally, setting $x:=x\sqcap z$, the last identity yields
\[
(y\sqcap(x\sqcap z))\sqcap((x\sqcap z)\sqcup z)\approx(x\sqcap z)\sqcap y,
\]
i.e.\
\[
(y\sqcap(x\sqcap z))\sqcap z\approx(x\sqcap z)\sqcap y,
\]
by (iv).
\end{enumerate}
\item In order to prove that $\mathbb P(\mathbf D)=(D,\leq,{}',0,1)$ is an orthoposet, just note that $x\sqcup0\approx x$ and $x\sqcup1\approx1$ entail that $0\leq x\leq1$, for any $x\in D$, respectively. Furthermore, since $\mathbf D$ is a directoid with an antitone involution (by (ii) and (v)), one has that $\mathbb P(\mathbf D)$ is a bounded poset with an antitone involution. Finally, if $x\leq x'$ then $x=0$, by (iv). The last observation entails that $L(x,x')=\{0\}$ as well as $U(x,x')=\{1\}$, for any $x\in D$.
\end{enumerate}
\end{proof}

By Lemma~\ref{lem: auxo-direct}, we conclude that any ortho-directoid is a bounded directoid with an antitone involution satisfying further conditions whose ratio will be clear soon. Moreover, $\mc{OD}\subseteq\mc{D}^{\mc P}_{c}$.

\begin{theorem}
Let $\mathbf D=(D,\sqcap,{}',0,1)$ be an ortho-directoid. Let us define a partial binary operation as follows, for any $x,y\in D$:
\begin{equation*}
x\oplus y:=
\begin{cases}
x\sqcup y         & \text{if }x\leq y', \\
\text{undefined,} & \text{otherwise}.
\end{cases}
\end{equation*}
Then $\mathcal O(\mathbf D):=(D,\oplus,{}',0,1)$ is an orthoalgebra.
\end{theorem}

\begin{proof}
Suppose that $x\oplus y$ is defined. Then $x\leq y'$ and $x\oplus y=x\sqcup y$. Since $'$ is an antitone involution on the poset induced by $\mathbf D$, $y\oplus x$ is defined and so
\[
x\oplus y=x\sqcup y=y\sqcup x=y\oplus x.
\]
Now, if $x\oplus y$ and $(x\oplus y)\oplus z$ are defined, then $x\leq y'$ and $x\sqcup y=x\oplus y\leq z'$. Hence, we have also $y\leq z'$ and so $y\oplus z=y\sqcup z$ is defined. By Definition~\ref{def: odir} (vi),
\[
x\sqcup (y\sqcup z)'=(x\sqcap y')\sqcup(y\sqcup(((x\sqcap y')\sqcup y)\sqcup z')')'=(y\sqcup(((x\sqcap y')\sqcup y)\sqcup z')')'=(y\sqcup z)'.
\]
Hence, we conclude that
\[
x\sqcup (y\sqcup z)=x\oplus(y\sqcup z)=x\oplus(y\oplus z)
\]
is defined and so, by Definition~\ref{def: odir} (iv), one has
\begin{align*}
(x\oplus y)\oplus z & =(x\sqcup y)\sqcup z=((x\sqcap y')\sqcup y)\sqcup(((x\sqcap y')\sqcup y)\sqcup z')'= \\
& =(x\sqcap y')\sqcup (y\sqcup(((x\sqcap y')\sqcup y)\sqcup z')')=x\sqcup (y\sqcup z)=x\oplus (y\oplus z).
\end{align*}
Furthermore, note that $x\leq x''$. Therefore, $x\oplus x'$ is defined and so $x\oplus x'=x\sqcup x'=1$, by Definition~\ref{def: odir} (ii). Now, if $x\leq y'$ and $(x'\sqcap y')'=x\sqcup y=1$, one has $x'\sqcap y'=1'=0$, by Lemma~\ref{lem: auxo-direct} (iii), upon noticing that $'$ is an involution. Hence,
\[
0=x'\sqcap y'=(x\sqcap y')'\sqcap y'
\]
and, by Lemma~\ref{lem: auxo-direct} (vi), $x\sqcap y'=y'$, i.e.\ $y'\leq x\leq y'$. We conclude that $y=x'$, since $'$ in an involution. The fact that, if $x\oplus1$ is defined, then $x=0$, trivially follows upon noticing that if $x\leq 1'=0$ (by Lemma~\ref{lem: auxo-direct} (iii)), then $x=0$ since $0$ is the least element in the poset induced by $\sqcap$, by Definition~\ref{def: odir} (iii). We conclude that $\mathcal O(\mathbf D):=(D,\oplus,{}',0,1)$ is an effect algebra. Moreover, $\mathcal O(\mathbf D)$ is also an orthoalgebra, since, by Lemma~\ref{lem: auxo-direct}, $\mathbb P(\mathbf D)=(D,\leq,{}',0,1)$ is an orthoposet.
%\begin{align*}
%
%\end{align*}
\end{proof}

Also, conversely, to every orthoalgebra an ortho-directoid can be assigned, see the next result.

\begin{theorem} Let $\mathbf A=(A,\oplus,{}',0,1)$ be an orthoalgebra. Let us set, for any $x,y\in A:$
\begin{equation*}
y\sqcup x=x\sqcup y:=
\begin{cases}
x\oplus y   & \text{if }x\leq y', \\
\max\{x,y\} & \text{if }x\not\leq y'\text{ and }(x\leq y\text{ or }y\leq x),  \\
z\in U(x,y) & \text{if }x\not\leq y'\text{ and }x\parallel y.
\end{cases}
\end{equation*}
Furthermore, define $x\sqcap y:=(x'\sqcup y')'$ for all $x,y\in A$. Then $\mathcal D(\mathbf A):=(A,\sqcap,{}',0,1)$ is an ortho-directoid. Moreover, $\mathcal O(\mathcal D (\mathbf A))=\mathbf A$.
\end{theorem}

\begin{proof}
Let us prove items (i) -- (vi) of Definition~\ref{def: odir}.
\begin{enumerate}[(i)]
\item Just note that $x\sqcap y\approx(x'\sqcup y')'\approx(y'\sqcup x')'\approx y\sqcap x$, by the definition of $\sqcup$.
\item Observe that, since $'$ is an involution, then $x\leq x''$. Therefore, $x\sqcup x'\approx x\oplus x'\approx1$.
\item $0\leq x'$ entails $0\sqcup x\approx0\oplus x\approx x$.
\item Note that $x\sqcap y'\leq y'$. Moreover,
\[
(x\sqcap y')\sqcup y\leq(((x\sqcap y')\sqcup y)\sqcup z)=(((x\sqcap y')\sqcup y)\sqcup z)''.
\]
Hence, we have
\begin{align*}
((x\sqcap y')\sqcup y)\sqcup(((x\sqcap y')\sqcup y)\sqcup z)' & \approx((x\sqcap y')\oplus y)\oplus(((x\sqcap y')\oplus y)\sqcup z)'\approx \\
                                                              & \approx(x\sqcap y')\oplus(y\oplus(((x\sqcap y')\oplus y)\sqcup z)').
\end{align*}
However,
\[
y\leq(((x\sqcap y')\oplus y)\sqcup z)
\]
and so
\[
y\oplus(((x\sqcap y')\oplus y)\sqcup z)'\approx y\sqcup(((x\sqcap y')\sqcup y)\sqcup z)'
\]
as well as
\[
(x\sqcap y')\leq(y\oplus(((x\sqcap y')\oplus y)\sqcup z)')'.
\]
Hence, we conclude
\[
(x\sqcap y')\oplus(y\oplus(((x\sqcap y')\oplus y)\sqcup z)')\approx(x\sqcap y')\sqcup(y\sqcup(((x\sqcap y')\sqcup y)\sqcup z)').
\]
\item Just note that $x\sqcap y\leq y=y''$ and
\[
x\sqcap y\leq((x\sqcap y)\sqcup y')''.
\]
We conclude that
\[
(x\sqcap y)\sqcup((x\sqcap y)\sqcup y')'\approx(x\sqcap y)\oplus((x\sqcap y)\oplus y')'\approx y.
\]
\item Since $x\sqcap y'\leq y'$ we have
\[
(x\sqcap y')\sqcup y\approx(x\sqcap y')\oplus y.
\]
Furthermore,
\[
y\leq(((x\sqcap y')\oplus y)\sqcup z)''=((x\sqcap y')\oplus y)\sqcup z.
\]
We have
\[
y\sqcup(((x\sqcap y')\sqcup y)\sqcup z)'\approx y\oplus(((x\sqcap y')\oplus y)\sqcup z)'.
\]
Moreover,
\[
((x\sqcap y')\oplus y)\leq(((x\sqcap y')\oplus y)\sqcup z)
\]
entails that
\[
((x\sqcap y')\oplus y)\oplus(((x\sqcap y')\oplus y)\sqcup z)'
\]
is defined as well as
\[
(x\sqcap y')\oplus(y\oplus(((x\sqcap y')\oplus y)\sqcup z)').
\]
Therefore,
\[
(x\sqcap y')\leq(y\oplus(((x\sqcap y')\oplus y)\sqcup z)')'.
\]
If
\[
(x\sqcap y')\leq y\oplus(((x\sqcap y')\oplus y)\sqcup z)'=(y\oplus(((x\sqcap y')\oplus y)\sqcup z)')'',
\]
then $x\sqcap y'=0$ and so the desired result follows trivially. Otherwise, we have
\begin{align*}
(x\sqcap y')\sqcup(y\sqcup(((x\sqcap y')\sqcup y)\sqcup z)')' & \approx\max\{x\sqcap y',(y\sqcup(((x\sqcap y')\sqcup y)\sqcup z)')'\}\approx \\
                                                              & \approx(y\sqcup(((x\sqcap y')\sqcup y)\sqcup z)')'.
\end{align*}
\end{enumerate}
The equality $\mathcal O(\mathcal D(\mathbf A))=\mathbf A$ is straightforward.
\end{proof}

The next corollary is immediate. It fully characterizes orthoposets induced by orthoalgebras in terms of their assigned directoids.

\begin{corollary}
Let $\mathbf P=(P,\leq,{}',0,1)$ be a bounded poset with an antitone involution. Then $\mathbf P$ is the bounded poset with antitone involution induced by an orthoalgebra if and only if there exists an assigned directoid $\mathbb D(\mathbf P)$ which is an ortho-directoid.
\end{corollary}

Moreover, since an orthoalgebra $\mb A$ is an orthomodular poset if and only if $x\oplus y=x\lor y$, for any pair of orthogonal elements $x,y\in A$ (see e.g.\ \cite{FoBe}), by Lemma \ref{lem3} one has also the following

\begin{corollary} Let $\mb A=(A,\oplus,',0,1)$ be an orthoalgebra. Then $\mb A$ is an orthomodular poset if and only if $\mc D(\mb A)$ satisfies $2.$ of Definition \ref{def: orthomodulardirectoid}.
\end{corollary}
In other words, we have a full equational characterization of the class of orthomodular posets within the class of orthoalgebras by means of their assigned varieties of bounded directoids with antitone involution which, in turn, are subvarieties of $\mc D^{\mc P}_{c}$. The results of this section, as far as we know, establish novel connections between orthomodular posets (lattices), orthoalgebras and paraorthomodular posets.
\section{Dedekind-MacNeille completions}\label{sec:4}

In this section we investigate bounded posets with an antitone involution whose Dedekind-MacNeille completion are paraorthomodular lattices.

It is well-known that every poset $(P,\leq)$ can be embedded into a complete lattice $\mathbf L$. We frequently take the so-called Dedekind-MacNeille completion $\BDM(P,\leq)$ for this $\mathbf L$.

Hence, let $\mathbf P=(P,\leq)$ be a poset. Put $\DM(\mathbf P):=\{L(A)\mid A\subseteq P\}$. Then $\BDM(\mathbf P):=(\DM(\mathbf P),\subseteq)$ is a complete lattice and $x\mapsto L(x)$ an embedding from $\mathbf P$ to $\BDM(\mathbf P)$ preserving all existing joins and meets, and an order isomorphism between the posets $\mathbf P$ and $(\{L(x)\mid x\in P\},\subseteq)$. We usually identify $P$ with $\{L(x)\mid x\in P\}$.

It is easy to see that if $B,C\subseteq P$ such that $B\leq C$ then $\bigvee_{\BDM(\mathbf P)}B=\bigwedge_{\BDM(\mathbf P)}C$ if and only if $L(C)\leq U(B)$.
%The lattice $\BDM(P)$ is called the Dedekind-MacNeille completion of $\mathbf P$.

By Banaschewski (\cite{bana}) and Schmidt (\cite{schmidt}) the Dedekind-MacNeille completion of a poset $\mathbf P$ is (up to isomorphism) any complete lattice $\mathbf L$ into which $\mathbf P$ can be supremum-densely and infimum-densely embedded (i.e.\, for every element $x\in L$ there exist $M,Q\subseteq P$ such that $x=\bigvee\varphi(M)=\bigwedge\varphi(Q)$, where $\varphi\colon{}P\to L$ is the embedding).

In what follows, if $\mathbf P=(P,\leq,{}',0,1)$ is a bounded poset with an antitone involution $'$ and $X,Y \subseteq P$, we define:
\begin{itemize}
\item $X':=\{x'\in P\mid x\in X\}$,
\item $X^\bot:=\{a\in P\mid a\leq x'\text{ for all }x\in X\}=L(X')$,
\item $X\vee_{\bot\bot}Y:=(X\cup Y)^{\bot\bot}$.
\end{itemize}

%Since $L(U(a,b),b)=L(b)$ we always have in a  sectionally pseudocomplemented poset that $b\leq  a*b$. 

A well-known result by Rie\v canov\'a (\cite{riecanova}) characterizes those effect algebras admitting an effect algebraic Dedekind-MacNeille completion in terms of strong $D$-continuity. Since we intend to characterize those involutive posets admitting a paraorthomodular Dedekind-MacNeille completion we will need a slightly weaker notion of weak $D$-continuity.

\begin{definition}\label{xDefinition2.1}
A bounded poset $(P,\leq,{}',0,1)$ with an antitone involution $'$ is called {\em weakly $D$-continuous} if and only if for all $B,C\subseteq P$ with $B\le C$ the following condition is satisfied:
\begin{list}{\labelitemi}{\leftmargin=0.5em}
\item[{\rm(WDC)}] If $\bigwedge\nolimits_{\mathbf P}(C\cup B')=0$ then $L(C)\leq U(B)$.
\end{list}
%\begin{enumerate}
%\item[{\rm(WDC)}] If $\bigwedge\nolimits_{\mathbf P}(C\cup B')=0$ then $L(C)\leq U(B)$.
%\end{enumerate}
\end{definition}

In the following, we establish a characterization of involutive bounded posets with para\-ortho\-mod\-u\-lar Dedekind-MacNeille completion.

\begin{theorem}\label{xTheorem4.4}
Let $\mathbf P=(P,\leq,{}',0,1)$ be a bounded poset with an antitone involution $'$. Then $\BDM(\mathbf P)$ is paraorthomodular if and only if $\mathbf P$ is weakly $D$-continuous.
\end{theorem}

\begin{proof}
First assume $\BDM(\mathbf P)$ to be paraorthomodular, $B,C\subseteq P$, $B\le C$ and $\bigwedge\nolimits_{\mathbf P}(C\cup B')=0$. Let $X:=\bigvee_{\BDM(\mathbf P)}B$ and $Y:=\bigwedge_{\BDM(\mathbf P)}C$. Then $X\subseteq Y$ and $X^\bot\wedge Y=0$. Since $\BDM(\mathbf P)$ is paraorthomodular we obtain $X=Y$. We conclude that $L(C)\leq U(B)$, i.e.\, $\mathbf P$ is weakly $D$-continuous. Conversely, assume $\mathbf P$ to be weakly continuous. Since $\BDM(\mathbf P)$ is a complete lattice with antitone involution $^\bot$ it is enough to check the following condition:
\begin{center}
If $X,Y\in\DM(\mathbf P)$, $X\subseteq Y$ and $X^\bot\wedge Y=0$ then $X=Y$.
\end{center}
We put $B:=X$ and $C:=U(Y)$. Then $B\leq C$ and $\bigwedge\nolimits_{\mathbf P}(C\cup B')=0$. We conclude from (WDC) that $X=\bigvee_{\BDM(\mathbf P)}B=\bigwedge_{\BDM(\mathbf P)}C=Y$. Hence $\BDM(\mathbf P)$ is paraorthomodular.
\end{proof}

\begin{corollary}\label{corstr}
Every weakly $D$-continuous involutive bounded poset is para\-ortho\-mod\-u\-lar. 
\end{corollary}

The next definition and theorem are suggested by a similar result of Niederle for Boolean posets (\cite[Theorem~17]{Niederle}).

\begin{definition}\label{doubly}
Let $\mathbf P=(P,\leq,{}',0,1)$ be a bounded poset with an antitone involution $'$. A subset $X$ of $P$ is called {\em involution-closed and doubly dense in $\mathbf P$} if the following conditions are satisfied:
\begin{enumerate}[{\rm(i)}]
\item $a=\bigvee_{\mathbf P}(L(a)\cap X)=\bigwedge_{\mathbf P}(U(a)\cap X)$ for all $a\in P$,
\item $X'\subseteq X$,
\item $0,1\in X$.
\end{enumerate}
\end{definition}

\begin{remark}\label{ddcd}
Recall that any involution-closed and doubly dense subset $X$ in $\mathbf P$ is a bounded poset with induced order and involution. Moreover, if $\mathbf P=(P,\leq,{}',0,1)$ is a bounded poset with an antitone involution $'$ then $P$ is an involution-closed and doubly dense subset in its Dedekind-MacNeille completion $\BDM(\mathbf P)$. This can be shown by the same arguments as in {\rm(\cite[Theorem 16]{Niederle})}, so we omit it.
\end{remark}

\begin{theorem}\label{Embnied}
{\em Embedding  theorem for weakly $D$-continuous paraorthomodular posets.} \\
Let $\mathbf P=(P,\leq,{}',0,1)$ be a bounded poset with an antitone involution $'$. Then $\mathbf P$ is a weakly $D$-continuous paraorthomodular poset if and only if there exists a complete paraorthomodular lattice $\mathbf L=(L,\vee,\wedge,{}',0,1)$ such that $P$ is an involution-closed and doubly dense subset of $L$.
\end{theorem}

\begin{proof}
From Theorem~\ref{xTheorem4.4} we know that every weakly $D$-continuous paraorthomodular poset $\mathbf P$ is an involution-closed and doubly dense subset in $\BDM(\mathbf P)$. Conversely, let $\mathbf P$ be an involution-closed and doubly dense subset of a complete paraorthomodular lattice $\mathbf L=(L,\wedge,\vee,{}',0,1)$. Then $\mathbf P$ is an involutive bounded poset. Assume $B,C\subseteq P$, $B\le C$ and $\bigwedge\nolimits_{\mathbf P}(C\cup B')=0$. Put $X:=\bigvee_{\mathbf L}B$ and $Y:=\bigwedge_{\mathbf L}C$. Then $X\leq Y$ and $X'\wedge Y=0$ (since $u\in P$, $u\leq X'\wedge Y$ implies $u\in L(C\cup B')$, i.e.\, $u=0$). Since $\mathbf L$ is paraorthomodular we obtain $X=Y$, i.e.\, $\mathbf P$ is weakly $D$-continuous and from Corollary~\ref{corstr} we have that it is also paraorthomodular.
\end{proof}

In \cite{Finch} Finch proved that an orthomodular poset admits an orthomodular DM-completion if and only if it satisfies the so-called $B$-property. Fazio, Ledda and Paoli gave in \cite[Proposition 2]{fazio} a new characterization of the $B$-property for orthoposets as follows.

\begin{proposition}
{\rm\cite[Proposition 2]{fazio}} An orthoposet $\mathbf P=(P,\leq,{}',0,1)$ has the $B$-property if and only if it satisfies for all $X,Y\subseteq P$

\begin{list}{\labelitemi}{\leftmargin=0.5em}
\item[{\rm(FLP)}] $L(X)\subsetneqq L(Y)$ implies $L(Y)\cap LU(X')\neq\{0\}$,
\end{list}
\end{proposition}

Using this concept we can prove

\begin{theorem}\label{parMNC}
Let $\mathbf P=(P,\leq,{}',0,1)$ be a bounded poset with an antitone involution $'$. Then $\BDM(\mathbf P)$ is paraorthomodular if and only if $\mathbf P$ satisfies {\rm(FLP)} for all $X,Y\subseteq P$.
\end{theorem}

\begin{proof}
If $\BDM(\mathbf P)$ is paraorthomodular, $X,Y\subseteq P$ and $L(X)\subsetneqq L(Y)$ then
\[
LU(X')\cap L(Y)=L(L(X)')\cap L(Y)=L(X)^\bot\cap L(Y)\neq\{0\}.
\]
If $\mathbf P$ satisfies (FLP) for all $X,Y\subseteq P$, $A,B\in\DM(\mathbf P)$, $A\subseteq B$ and $A^\bot\cap B=\{0\}$ then
\[
B\cap LU(U(A)')=B\cap LUL(A')=B\cap L(A')=B\cap A^\bot=\{0\}
\]
and hence $A=B$ proving paraorthomodularity of $\BDM(\mathbf P)$.
\end{proof}

Summarizing the previous results we obtain

\begin{corollary}\label{cparMNC}
For a bounded poset $\mathbf P=(P,\leq,{}',0,1)$ with an antitone involution $'$ the following conditions are equivalent:
\begin{enumerate}[{\rm(i)}]
\item $\BDM(\mathbf P)$ is paraorthomodular,
\item $\mathbf P$ is weakly $D$-continuous,
\item $\mathbf P$ satisfies {\rm(FLP)} for all $X,Y\subseteq P$.
\end{enumerate}
\end{corollary}

{\bf Compliance with Ethical Standards.} The authors declare that they have no conflict of interest.

{\bf Acknowledgement.} Support of the research of the first and third author by the Austrian Science Fund (FWF), project I~4579-N, and the Czech Science Foundation (GA\v CR), project 20-09869L, entitled ``The many facets of orthomodularity'', as well as by \"OAD, project CZ~02/2019, entitled ``Function algebras and ordered structures related to logic and data fusion'', and, concerning the first author, by IGA, project P\v rF~2020~014, is gratefully acknowledged. D.~Fazio and A.~Ledda gratefully acknowledge the following funding sources: Regione Autonoma della Sardegna, within the project \textquotedblleft Per un'estensione semantica della Logica Computazionale Quantistica - Impatto teorico e ricadute implementative\textquotedblright, RAS: SR40341; MIUR, project PRIN 2017 \textquotedblleft Theory and applications of resource sensitive logics\textquotedblright, CUP: 20173WKCM5; \textquotedblleft Logic and cognition. Theory, experiments, and applications\textquotedblright, CUP: 2013YP4N3; Fondazione di Sardegna, within the project \textquotedblleft Resource sensitive reasoning and logic\textquotedblright, CUP: F72F20000410007. Research of the last author was supported by the project entitled ``Group Techniques and Quantum Information'', No.~MUNI/G/1211/2017, by Masaryk University Grant Agency (GAMU).

%The authors  would  like to express their gratitude to the anonymous reviewer whose thorough and detailed comments lead to a considerable improvement of this paper. 

Authors' addresses:

Ivan Chajda \\
Palack\'y University Olomouc \\
Faculty of Science \\
Department of Algebra and Geometry \\
17.\ listopadu 12 \\
771 46 Olomouc \\
Czech Republic \\
ivan.chajda@upol.cz

Davide Fazio \\
A.Lo.P.Hi.S Research Group\\
%Department of Pedagogy, Psychology and Phylosophy\\
University of Cagliari\\
Via Is Mirrionis, 1\\
09123, Cagliari, Italy\\
dav.faz@hotmail.it

Helmut L\"anger \\
TU Wien \\
Faculty of Mathematics and Geoinformation \\
Institute of Discrete Mathematics and Geometry \\
Wiedner Hauptstra\ss e 8-10 \\
1040 Vienna \\
Austria, and \\
Palack\'y University Olomouc \\
Faculty of Science \\
Department of Algebra and Geometry \\
17.\ listopadu 12 \\
771 46 Olomouc \\
Czech Republic \\
helmut.laenger@tuwien.ac.at

Antonio Ledda \\
A.Lo.P.Hi.S Research Group\\
%Department of Pedagogy, Psychology and Phylosophy\\
University of Cagliari\\
Via Is Mirrionis, 1\\
09123, Cagliari, Italy\\
antonio.ledda@unica.it

Jan Paseka \\
Masaryk University Brno \\
Faculty of Science \\
Department of Mathematics and Statistics \\
Kotl\'a\v rsk\'a 2 \\
611 37 Brno \\
Czech Republic \\
paseka@math.muni.cz
\end{document}